\documentclass[11pt,letterpaper,reqno]{amsart}

\usepackage{amssymb,amsmath,amsthm,latexsym,amscd,relsize}
\usepackage{epsfig}
\usepackage{psfrag}
\usepackage{graphicx}
\usepackage{bbold}
\usepackage{tikz}
\usepackage{cleveref}
\usepackage{float}
\usepackage{filecontents}
\usepackage{scalerel}

\newtheorem{Theorem}{\bf Theorem}
\newtheorem{lemma}[Theorem]{\bf Lemma}
\newtheorem{proposition}[Theorem]{\bf Proposition}

\newtheorem{definition}[Theorem]{\bf Definition}

\newtheorem{remark}[Theorem]{\bf Remark}
\newtheorem{theorem}[Theorem]{\bf Theorem}

\def\scfig #1 #2 {\resizebox{#2}{!}{\includegraphics{#1}}}

\newcommand{\be}{\begin{equation}}
\newcommand{\ee}{\end{equation}}

\def\hpic #1 #2 {\mbox{$\begin{array}[c]{l} 
\epsfig{file=#1,height=#2}\end{array}$}}
\def\wpic #1 #2 {\mbox{$\begin{array}[c]{l} 
\epsfig{file=#1,width=#2}\end{array}$}}

\begin{document}
\title{Intermediate planar algebra revisited, II}
\author{Keshab Chandra Bakshi}
\author{Sruthymurali}
\address{Chennai Mathematical Institute, Chennai, India,}
\email{bakshi209@gmail.com, kcbakshi@cmi.ac.in, sruthy92smk@gmail.com, sruthy@cmi.ac.in}
\date{04 October, 2021}
\subjclass[2010]{46L37}
\begin{abstract} 
	We describe the subfactor planar algebra of an intermediate subfactor $N\subset Q \subset M$ of an extremal subfactor $N\subset M$ of finite Jones index which is not necessarily irreducible. 
\end{abstract}
\maketitle

\section{Introduction}
Given the fact that modern subfactor theory, as pioneered by Vaughan Jones, deals with the relative position of a subfactor inside an ambient factor, it is a very fundamental question to consider relative positions of an intermediate subfactor and this turns out to be a very active area of research. Given a finite-index irreducible subfactor  $N\subset M$ (that is, $N^{\prime}\cap M=\mathbb{C}$) and an intermediate subfactor $Q$, one may explicitly describe (as in \cite{Bakshiintermediate}, see also \cite{BhaLa})  the planar algebra of $N \subset Q$ in terms of the planar algebra of $N\subset M$. The proof in \cite{Bakshiintermediate} is technical which involves lots of beautiful pictorial calculations involving the so-called `biprojections’. We must mention that in the proof we have crucially used the fact that $N\subset M$ is irreducible. However, in the non-irreducible case the description of the planar algebra of $N\subset Q$ has not yet appeared in the literature.  In this paper we consider an intermediate subfactor $N\subset Q\subset M$ of an extremal subfactor $N\subset M$ of finite Jones index which is not necessarily irreducible and  describe the subfactor planar algebra of $N\subset Q$ (which we denote by $P^{N\subset Q}$) in terms of  the subfactor planar algebra $P^{N\subset M}$ in such a way that in the irreducible case, it recovers the description of $P^{N\subset Q}$ as expounded in \cite{Bakshiintermediate}.

We briefly mention related work. In \cite{hartglass17}, Hartglass has constructed an object called an  N-P-M planar algebra which is an algebra over an operad of three-shaded tangles  and shown that if $\mathcal{Q}$ is the standard invariant of $N\subset M$  that contains an intermediate subfactor $P$, then $\mathcal{Q}$ can be faithfully realized inside a natural N-P-M planar algebra  $\mathcal{P}$ associated to the triple $N\subset P \subset M$.
It is also worth mentioning that D. Bisch gave a partial description of the standard invariant of $N \subset Q$ in \cite{Bi94} by giving the standard invariant of the inclusion $N \subset Q_1$, where $Q_1$ is the first step in the basic construction of $N\subset Q$.
Our main result is an explicit description of the  planar algebra of $N\subset Q$ in terms of the original planar algebra. 
We fix an arbitrarily chosen minimal projection $f$ in $Q^{\prime}\cap M$ which gives rise to a projection $f_n$ in $P^{N\subset M}_n$ and a new partially labelled tangle $fTf$ associated to any  tangle $T$. The biprojection $q$ (namely, the Jones projection $e_Q$) corresponding to the intermediate
subfactor $Q$ and the projection $f$ yield naturally a mapping $F = \{F_m\}$ from tangles of any colour (say $m$) to partially labelled
tangles of the same colour and carefully chosen scalar-valued
functions $\alpha_{N\subset Q\subset M}$ and $c$ defined on the collection of all tangles (see Definition \ref{defalpha}), such that $P^{N \subset Q}$ may be identified with a planar algebra, call it $P^\prime$, with $P'_n = range (Z_{F(fI^n_nf)}^{N \subset M})$, where $I^n_n$ is the identity tangle of colour $n$ and
the multilinear map $Z^{\prime}$ associated to a tangle $T:= T^{k_0}_{k_1,\cdots, k_b}$ is given by
\begin{equation*} 
Z^{\prime}_T = (\text{tr}f)^{c(T)}\alpha_{N\subset Q\subset M}(T)Z^{N \subset M}_{F(fTf)},
\end{equation*} with inputs from $P^\prime$. In order to verify that  $(P^{\prime},Z^{\prime}_T)$ is indeed a planar algebra we need to check that  the operation of tangles is compatible with composition of tangles and to this end we  verify the following crucial `co-cycle type' equation holds for any two tangles $T$ and $\tilde{T}$ so that $T\circ \tilde{T}$ makes sense:
\begin{equation}\label{cocycle}
 Z^{N\subset M}_{F(T\circ \tilde{T})}=\frac{\Big((\text{tr}f)^{(c(T)+c(\tilde{T}))}\alpha_{N\subset Q\subset M}(T) \alpha_{N\subset Q \subset M}(\tilde{T})\Big)}{\Big((\text{tr}f)^{c(T\circ \tilde{T})}\alpha_{N\subset Q\subset M}(T\circ \tilde{T})\Big)}
Z^{N\subset M}_{F(T)\circ F(\tilde{T})}.
\end{equation}

We divide the proof of the verification of the Equation \ref{cocycle} into two cases as follows:
\begin{center}
	Case (I): $Q^\prime \cap M= \mathbb{C}$;  and Case (II): $Q^\prime \cap M\neq \mathbb{C}.$
\end{center}

\noindent In view of the fact that in Case (I) we must have $f=1$ it is easy to see that  $(P^{\prime},Z^{\prime}_T)$ coincides with the description of the intermediate planar algebra for the irreducible case as in \cite{Bakshiintermediate}.  
The verification of the  Equation \ref{cocycle}  involves a suitable adaptation of Theorem 3.4 in \cite{Bakshiintermediate} using a cute pictorial relation involving $q$. For the proof of Case (II) we have taken a two-fold strategy. First for the chosen minimal projection $f$  in $Q^\prime \cap M$  we consider the inclusions $Nf\subset Qf\subset fMf$ and then we observe that the description of $P^{Nf\subset Qf}$ reduces to the Case (I). In other words, using Case (I) we may describe $P^{Nf\subset Qf}$  in terms of $P^{Nf\subset fMf}$. As a final strategy, we simply observe that $P^{Nf\subset Qf}$ and $P^{N\subset Q}$ are isomorphic as planar algebras and recall from Corollary 4.2.14 in \cite{Joplanar1} that the planar algebra of the subfactor $Nf\subset fMf$ is isomorphic to the reduced planar algebra $fP^{N\subset M}f$. Therefore, we obtain a description of $P^{Nf\subset Qf}$ (and hence of $P^{N\subset Q})$ in terms of $P^{N\subset M}$. 
\section{Notation and basic facts}
In this paper, all factors we will be considering are of type ${II}_1$  and all subfactors $N\subset M$ will have finite Jones index $[M:N]$. By $\text{tr}_M$ we mean the unique normal faithful trace defined on $M$. $E^M_N$ denotes the trace preserving conditional expectation from $M$ onto $N$; we may omit the subscript $M$ and instead write $E_N$ and tr if it is clear from the context.
We will assume throughout that the reader is familiar with planar algebras introduced by Jones in \cite{Joplanar1}. To fix the notations and definitions for the version
of planar algebras that we use, we refer to \cite{KodSun}. We will briefly describe the notations here for the reader's convenience. Let $P=(P_k)_{k\geq 0}$ be a planar algebra where $P_k$ denotes the $k$-box space $N^\prime \cap M_{k-1}$, $\delta = [M:N]^{-1/2}$ and write $Z_T$ for the multilinear operator corresponding to a planar tangle $T$.
We dispense with shading figures since the shading is uniquely determined by the sub- and superscripts of the tangle.

Now given a planar algebra $P$ and an $f \in P_1$ a projection, we can produce a new planar algebra $fPf$ called the \textit{reduced planar algebra} which we briefly describe now. More details can be found in \cite{Joplanar1}. First define the projections $f_k \in P_k$ as in Figure \ref{fig:projectionfk}. Given a planar tangle $T$, define the partially labelled tangle $fTf$ by inserting $\begin{minipage}{.05\textwidth}
\centering
\psfrag{f}{$f$}
\includegraphics[scale=.45]{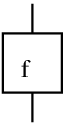}
\end{minipage}
$
in each string of $T$. Now let us describe the planar algebra $fPf$ as follows. First define the spaces of $fPf$ as $(fPf)_k=f_kP_kf_k$ and the tangle action by $Z_T^{fPf}=Z_{fTf}^P$. It is easy to check that $\displaystyle{(fPf,Z^{fPf}_T)}$ is a subfactor planar algebra and we have the following:

\begin{proposition}[\cite{Joplanar1}, Corollary 4.2.14]
	Let $N \subset M$ be an extremal ${II}_1$ subfactor with $[M:N]< \infty$ and $f$ a projection in $N'\cap M$. Then the reduced planar algebra $fP^{N \subset M}f$ is naturally isomorphic to the planar algebra of the reduced subfactor $Nf\subset fMf$.
\end{proposition}
\begin{figure}[!h]
	\begin{center}
		\psfrag{f}{$f$}
				\psfrag{dots}{$\cdots$}
		\resizebox{6.0cm}{!}{\includegraphics{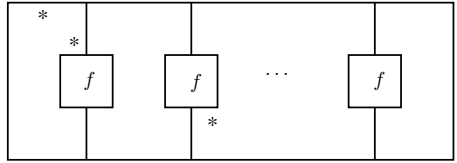}}
	\end{center}
	\caption{The projection $f_k \in P_k$.}
	\label{fig:projectionfk}
\end{figure}

It is well-known from \cite{Bi1,La1,BiJo2} that if $N^\prime \cap M = \mathbb{C}$, there is a bijective correspondence between biprojections $q$ (corresponding to the Jones projection of $L^2(M)$ onto $L^2(Q)$) and the intermediate subfactor $Q$, where  $N \subset Q \subset M$. More precisely, we have the following (reformulation of) Theorem 3.2 of \cite{Bi1}. 
\begin{theorem}[\label{Bisch}\cite{Bi1}, \cite{La1}, \cite{BiJo2}]
	Let $N \subset M$  be an extremal ${\rm II}_1$ subfactor and $P=P^{(N\subset M)}$ be the planar algebra of $N \subset M$. Suppose there exists an intermediate subfactor $Q$, $N \subset Q \subset M$ and $q \in P_2$ denotes the biprojection corresponding to $Q$. Then $q$ satisfies the relations in Figure \ref{fig:biprojection}
	\begin{figure}[!h]\label{fig:biprojection}
\begin{align*}
&\begin{tikzpicture}
\node at (-1,.25) {$(a)$};
\path [fill=white] (.3,-.5) rectangle (.7,1);
\draw (.3,-.5)--(.3,1);
\draw (.7,1)--(.7,-.5);
\draw [fill=white](0,0) rectangle (1,.5);
\node at (0,.25) [left] {$*$};
\node at (.5,.25) {$q$};
\node at (1.1,.25)[right] {$=$};
\path [fill=white] (2.3,-.5) rectangle (2.7,1);
\draw (2.3,-.5)--(2.3,1);
\draw (2.7,1)--(2.7,-.5);
\draw [fill=white] (2,0) rectangle (3,.5);
\node at (2.5,.25) {$q$};
\node at (3,.25)[right]{$*$};
\node at (5,.25) {$(b)$};
\path [fill=white] (6.3,-.5)--(6.3,.6) arc [radius=.2,start angle=180,end angle=0]--(6.7,-.5);
\draw (6.3,-.5)--(6.3,.6);
\draw (6.7,-.5)--(6.7,.6);
\draw (6.7,.6) arc [radius=.2,start angle=0, end angle=180];
\draw [fill=white](6,0) rectangle (7,.5);
\node at (6,.25)[left] {$*$};
\node at (6.5,.25) {$q$};
\node at (7.5,.25) {$=$};
\path [fill=white] (8.3,-.5)--(8.3,0) arc[radius=.2,start angle=180, end angle=0]--(8.7,-.5);
\draw (8.3,0) arc[radius=.2,start angle=180, end angle=0];
\draw (8.3,0)--(8.3,-.5);
\draw (8.7,0)--(8.7,-.5);
\end{tikzpicture}\\
&\begin{tikzpicture}
\node at (-1,.25) {$(c)$};
\path [fill=white](.3,-.5) rectangle (.7,1);
\path [fill=white] (.7,1)--(.7,.6) arc[radius=.3,start angle=180, end angle=0]--(1.3,.6)--(1.6,.6)--(1.6,1);
\path [fill=white] (.7,-.5)--(.7,-.1) arc[radius=.3,start angle=180, end angle=360]--(1.3,.6)--(1.6,.6)--(1.6,-.1)--(1.6,-.5);
\draw (.3,-.5)--(.3,1);
\draw (.7,.6)--(.7,-.1);
\draw (.7,.6) arc[radius=.3, start angle=180, end angle=0];
\draw (1.3,.6)--(1.3,-.1);
\draw (.7,-.1) arc[radius=.3, start angle=180, end angle=360];
\draw [fill=white] (0,0) rectangle (1,.5);
\node at (0,.25)[left] {$*$};
\node at (.5,.25) {$q$};
\node at (1.8,.25) [right] {$=c$};
\path [fill=white](3,1) rectangle (3.5,-.5);
\draw (3,1)--(3,-.5);
\node at (3,.25) [left]{$*$};
\end{tikzpicture}\\
&\begin{tikzpicture}
\node at (-1,.25) {$(d)$};
\path [fill=white] (.3,1)--(.3,-2)--(.7,-2)--(.7,-1.5) arc[radius=.3,start angle=180, end angle=90]--(2.2,-1.2) arc[radius=.3,start angle=90, end angle=0]--(2.5,-2)--(2.8,-2)--(2.8,1)--(2.5,1)--(2.5,-.5) arc[radius=.3,start angle=0,end angle=-90]--(1,-.8) arc[radius=.3,start angle=270,end angle=180]--(.7,1);
\draw (.3,1)--(.3,-2);
\draw (.7,1)--(.7,-.5);
\draw (.7,-.5) arc [radius=.3,start angle=180, end angle=270];
\draw (1,-.8)--(2.2,-.8);
\draw (2.2,-.8) arc[radius=.3,start angle=270, end angle=360];
\draw (2.5,-.5)--(2.5,1);
\draw (2.2,-1.2) arc[radius=.3,start angle=90, end angle=0];
\draw (2.2,-1.2)--(1,-1.2);
\draw (1,-1.2) arc[radius=.3,start angle=90,end angle=180];
\draw (.7,-1.5)--(.7,-2);
\draw (2.5,-1.5)--(2.5,-2);
\draw [fill=white] (0,0) rectangle (1,.5);
\draw [fill=white] (1.5,-.5) rectangle (2,-1.5);
\node at (0,.25)[left] {$*$};
\node at (1.75,-1.5)[below] {$*$};
\node at (.5,.25) {$q$};
\node at (1.75,-1){$q$};
\node at (3.5,-.5){$=$};
\path [fill=white] (4.8,1)--(4.8,-2)--(5.2,-2)--(5.2,-.5) arc[radius=.3,start angle=180,end angle=90]--(6.7,-.2) arc[radius=.3,start angle=90,end angle=0]--(7,-2)--(7.3,-2)--(7.3,1)--(7,1)--(7,.5) arc[radius=.3,start angle=0, end angle=-90]--(5.5,.2) arc[radius=.3,start angle=270,end angle=180]--(5.2,1);
\draw (4.8,-2)--(4.8,1);
\draw (5.2,-2)--(5.2,-.5);
\draw (5.2,-.5) arc[radius=.3, start angle=180, end angle=90];
\draw (5.5,-.2)--(6.7,-.2);
\draw (6.7,-.2) arc[radius=.3,start angle=90, end angle=0];
\draw (7,-.5)--(7,-2);
\draw (7,1)--(7,.5);
\draw (7,.5) arc[radius=.3,start angle=0, end angle=-90];
\draw (6.7,.2)--(5.5,.2);
\draw (5.5,.2) arc[radius=.3,start angle=270, end angle=180];
\draw (5.2,.5)--(5.2,1);
\draw [fill=white] (4.5,-1.5) rectangle (5.5,-1);
\draw [fill=white] (6,.5) rectangle (6.5,-.5);
\node at (4.5,-1.25) [left] {$*$};
\node at (6.25,-.5)[below] {$*$};
\node at (5,-1.25) {$q$};
\node at (6.25,0){$q$};
\end{tikzpicture}
\end{align*}
\caption{Biprojection relations.}

\end{figure}
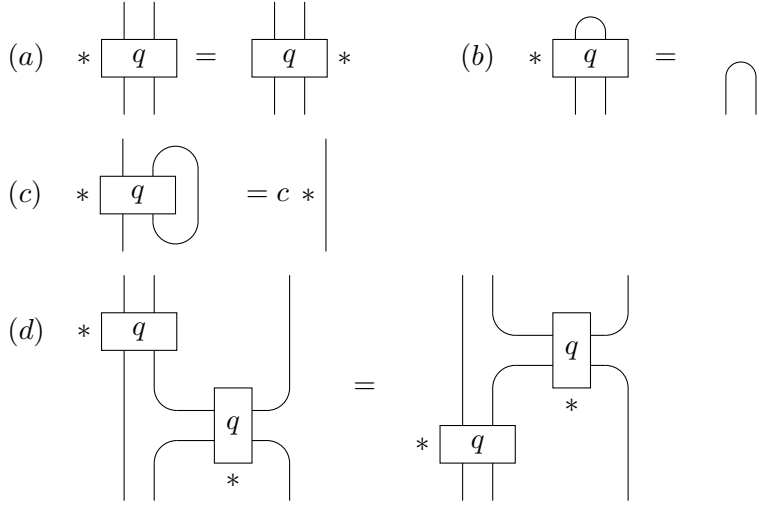
with $c = [M:N]^{1/2}[M:Q ]^{-1}$. Furthermore, in the case $N'\cap M=\mathbb{C}$, the converse is also true. More precisely, $q\in P_2$ satisfying the relations (a)-(d) in Figure \ref{fig:biprojection} implies the 
existence of an intermediate subfactor $Q$, $N \subset Q \subset M$ corresponding to $q$.
\end{theorem}
Throughout the paper let us fix a ${II}_1$ subfactor $N\subset M$ having an intermediate $Q$ with the corresponding biprojection $q \in P_2$ (i.e $q=e_Q$) where $P=P^{N\subset M}$. Our goal is to determine $P^{N\subset Q}$ in terms of the planar algebra $P^{N\subset M}$.  Now let $f \in Q^\prime \cap M$ be a minimal projection. Consider the projection $f_2$ defined as in the Figure \ref{fig:projectionfk}. Then we have,
\begin{lemma}{\label{biprojection}}
	Let $N\subset Q\subset M$ be an intermediate subfactor with the corresponding biprojection $q$ and $f\in Q^\prime \cap M$ be a minimal projection.  Then $\frac{1}{\text{tr}f}f_2qf_2$ is the biprojection corresponding to the intermediate subfactor $Nf\subset Qf\subset fMf$.
	\end{lemma}
\begin{proof}
	This essentially follows from Lemma 2.2 from \cite{Bi1}. Indeed in that lemma replace $N$ by $Q$, $p$ by $f$, $q$ by $\tilde{f}=
	$\begin{minipage}{.1\textwidth}
		\centering
		\psfrag{f}{\tiny $f$}
		\includegraphics[scale= .25]{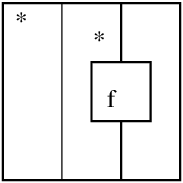}
	\end{minipage} and hence $M_1$ is replaced by $Q_1$ where 
	$Q_1$ is the basic construction of $Q \subset  M$. It is easy to observe that $f\tilde{f}=\tilde{f}f=f_2$.
Then $Qf_2\subset (fMf)\tilde{f}\subset f_2Q_1f_2$ is the basic construction of ($Qf_2\subset (fMf)\tilde{f})\cong(Qf\subset fMf)$ with the corresponding biprojection $\frac{1}{\text{tr} f}f_2qf_2$.
\end{proof}
\begin{lemma}{\label{subfactoriso}}
The subfactor $Nf\subset Qf$ is isomorphic to $N \subset Q$.	
\begin{proof}
	It is easy to see that the map $x \mapsto xf$ is the required isomorphism. We leave the details to the interested reader.
\end{proof}
\end{lemma}
\begin{lemma}\label{minimalprojection}
	Let $A\subset B$ be an inclusion of von Neumann algebras. Then $q(A^\prime \cap B)q=(qA)^\prime\cap qBq$ for any projection $q \in A^\prime \cap B$.
\end{lemma}
\begin{proof}
	This is a well known fact (see for instance Fact 2.2 of \cite{BhaLa}).
\end{proof}
Let $E_n$ denote the tangle as in Figure \ref{fig:definingtangle}.
\begin{figure}[!h]
	\begin{center}
		\psfrag{d1}{$D_1$}
		\psfrag{d2}{$D_2$}
		\psfrag{dn}{$D_n$}
		\psfrag{dn+1}{$D_{n+1}$}
		\psfrag{dots}{$\cdots$}
		\resizebox{7.0cm}{!}{\includegraphics{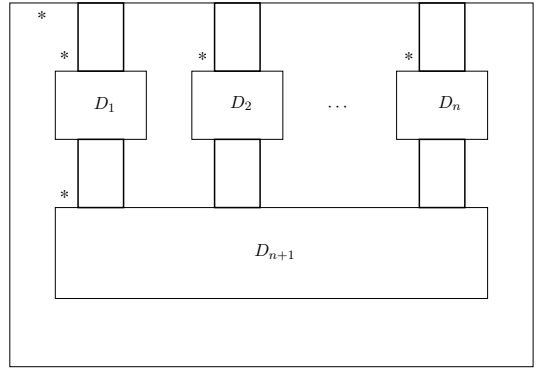}}
	\end{center}
	\caption{The defining tangle $E_n$.}
	\label{fig:definingtangle}
\end{figure}
	For a tangle $T$, define the partially labelled tangle $F_n(T)$ by inserting $\frac{1}{\text{tr}f}f_2qf_2$ in $D_1, \cdots D_n$ and $T$ in $D_{n+1}$ of $E_n$ respectively. That is, 
$$F_n(T)=E_n \circ_{(D_1,\cdots,D_n,D_{n+1})}(\frac{1}{\text{tr}f}f_2qf_2,\cdots,\frac{1}{\text{tr}f}f_2qf_2,T).$$
We write $E$ in place of $E_n$ if it is clear from the context. Now for $x \in P_n$, define $F_n(x)=Z_{F_n(I^n_n)}^P(x)$. Pictorially we have $F_n(x)$ as in Figure \ref{fig:fnx}.
\begin{figure}[!h]
	\begin{center}
		\psfrag{f}{$f$}
		\psfrag{q}{\huge $q$}
	\psfrag{x}{\huge $x$}
	\psfrag{trf}{\huge $\frac{1}{(\text{tr}f)^n}$}
		\psfrag{dots}{$\cdots$}
		\resizebox{7.5cm}{!}{\includegraphics{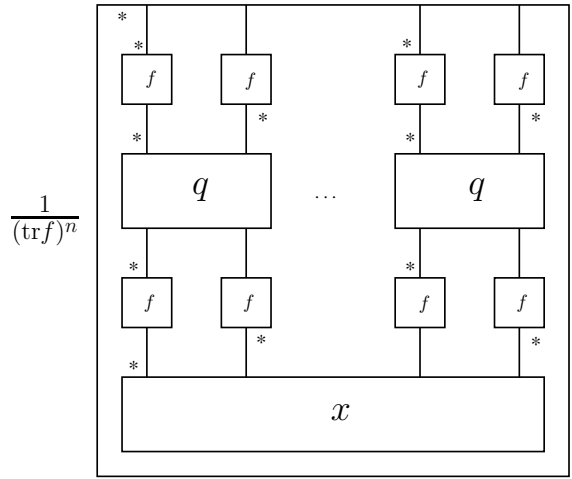}}
	\end{center}
	\caption{The element $F_n(x)$.}
	\label{fig:fnx}
\end{figure}
Thus $F_n$ can be viewed as a map from $P_n$ into $P_n$ given by $F_n(x)=Z_{E_n}(\frac{1}{\text{tr}f}f_2qf_2\otimes \cdots\otimes \frac{1}{\text{tr}f}f_2qf_2\otimes x)$.
\section{Main result}

Throughout this section we fix an extremal finite index subfactor $N\subset M$ having an intermediate subfactor $Q$. The main result of this section is the description of the planar algebra $P^{N\subset Q}$ in terms of $P=P^{N\subset M}$ by explicitly specifying the spaces and tangle action in terms of $P$. We begin by recalling a definition and proposition from \cite{Bakshiintermediate}.



\begin{definition}\label{defalpha}{\normalfont(Definition 3.1, \cite{Bakshiintermediate})}
Let $ T $ be a $k$-tangle with $ b\geq 1 $ internal discs ${D_1,\dots D_b}$ of colours ${k_1,\dots k_b}$. Then define $\alpha_{N\subset Q\subset M}(T) = [M:Q]^{\frac{1}{2}c(T)}$, where
\[c(T)=(\lceil k_0/2 \rceil+\lfloor k_1/2 \rfloor+\dots +\lfloor k_b/2 \rfloor)-l(T) \]
with $l(T)$ being the number of closed loops after capping the black intervals of the external disc  of $T$ and cupping the black intervals of all internal discs of $T$. 
\end{definition}
\begin{proposition}{\normalfont (Proposition 3.2, \cite{Bakshiintermediate})}\label{alphas}
	If $T = T^{k_0}_{k_1,\cdots,k_b}$ and $\tilde{T} = \tilde{T}^{\tilde{k}_0}_{\tilde{k}_1,\cdots,\tilde{k}_{\tilde{b}}}$ are tangles with discs of indicated colours such that $\tilde{k}_0=k_i$ for some $1 \leq i \leq b$, then
	\[\frac{\alpha_{N\subset Q\subset M}(T)\alpha_{N\subset Q\subset M}(\tilde{T})}{\alpha_{N\subset Q\subset M}(T\circ_i\tilde{T})} =  [M:Q]^{\frac{1}{2}(\tilde{k_0} - l(T) - l(\tilde{T}) + l(T\circ_i\tilde{T}))}.\]
\end{proposition}
\bigskip
 Recall that for a planar tangle $T$ and an arbitrarily chosen minimal projection $f \in Q^\prime \cap M$,  $F(T)$ denotes the partially labelled tangle obtained from $T$ by ``surrounding it with $\frac{1}{(\text{tr}f)}f_2qf_2$'' where $f_2$ is the projection given in the Figure \ref{fig:projectionfk} and $fTf$ denotes the partially labelled tangle obtained by inserting $\begin{minipage}{.05\textwidth}
 \centering
 \psfrag{f}{$f$}
 \includegraphics[scale=.45]{projectionf.eps}
 \end{minipage}
 $
 in each string of $T$.
\begin{definition}
	Define $P_n^\prime \subset P_n$ by $P_n^\prime =F_n(f_nP_nf_n)$ and $$Z_T^{P^\prime}=(\text{tr}f)^{c(T)}\alpha_{N\subset Q\subset M}(T)Z^P_{F(fTf)}|_{P^\prime}.$$
	
\end{definition}
\begin{remark}\label{example}
	It is easy to observe that $Z^P_{fF(T)f}=Z^P_{F(fTf)}$ with inputs coming from $P^\prime$. We illustrate this fact with an example and it should be clear that the proof of the general case is similar. 
	Let $T$ be the tangle given in Figure \ref{fig:example}.
	
	\begin{figure}[!h]
		\begin{center}
			\resizebox{5cm}{!}{\includegraphics{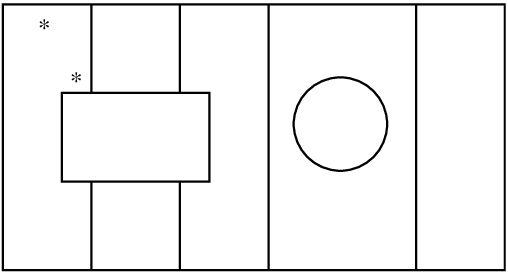}}
		\end{center}
		\caption{}
		\label{fig:example}
	\end{figure}
The tangles $fF(T)f$ and $F(fTf)$ are shown in Figure \ref{fig:examplefFf}. Now the assertion is clear from the fact that $f$ is a projection in $P_1$ and inputs are coming from $P^\prime$.
\begin{figure}[!h]
	\begin{center}
		\psfrag{*}{$*$}
		\psfrag{f}{$f$}
		\psfrag{f2qf2}{$f_2qf_2$}
		\psfrag{S}{\huge $fF(T)f$}
		\psfrag{T}{\huge $F(fTf)$}
		\psfrag{s}{\huge $\frac{1}{(\text{tr}f)^4}$}
		\resizebox{14cm}{!}{\includegraphics{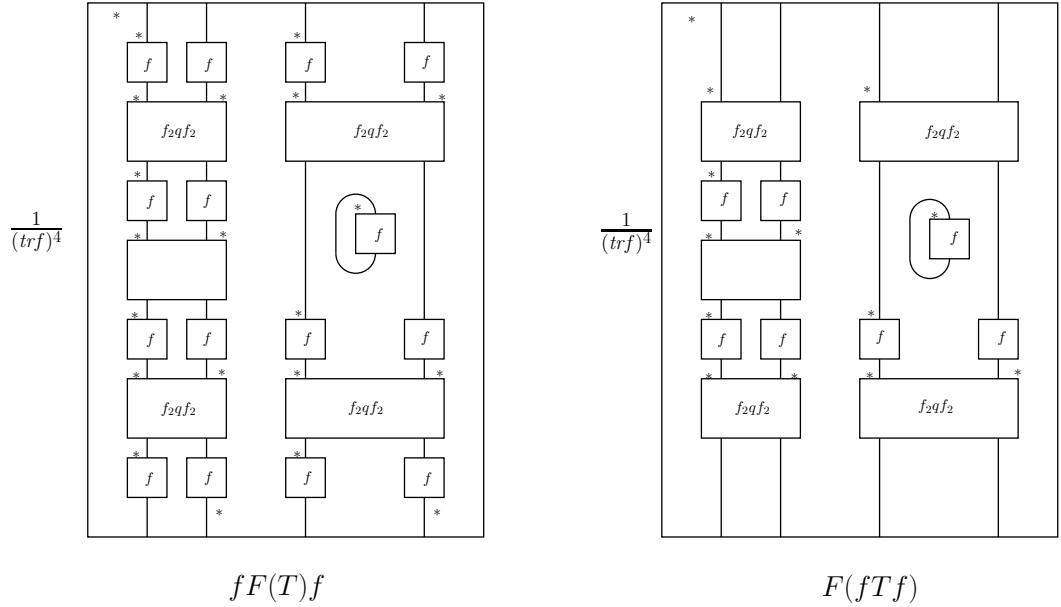}}
	\end{center}
	\caption{The tangles $fF(T)f$ and $F(fTf)$.}
	\label{fig:examplefFf}
\end{figure}
	\end{remark}
Now we prove the main theorem of this article.
\begin{theorem}\label{main1}
Keeping the foregoing notations, 
$(P^\prime, Z_T^{P^\prime}|_{P^\prime})$ is a subfactor planar algebra which is isomorphic to $P^{N\subset Q}$. 	
\end{theorem}
\begin{proof}
	Our proof strategy is as follows. We divide the proof into two cases (I) $Q^\prime \cap M=\mathbb{C}$ and (II) $Q^\prime \cap M\neq \mathbb{C}$. The proof of the Case (I) is  similar to the proof of the  irreducible case as in \cite{Bakshiintermediate} with appropriate  modifications, and for the proof of Case (II) we will choose a minimal projection $f\in Q^\prime \cap M$ and consider the intermediate $Nf\subset Qf\subset fMf$ which reduces to Case (I), thereby finishing the proof.
	
	\smallskip\noindent
	
	{Case (I):} In this case $f=1$ and hence the operation $F(T)$ for a tangle $T$ is nothing but surrounding $T$ by ``$q$''. So our description of the spaces $P^\prime_n=F_n(P_n)$ and the tangle action $Z_T^{P^\prime}=\alpha_{N\subset Q\subset M}(T)Z^P_{F(T)}|_{P^\prime}$ where $\alpha_{N\subset Q\subset M}(T)=[M:Q]^{\frac{1}{2}c(T)}$ coincides with the description in the irreducible case. 
	Therefore, in view of Proposition \ref{alphas},  the key ingredient in proving that $(P^\prime, Z_T^{P^\prime})$ to be a planar algebra is the verification of the Equation \ref{mainequation} (see \cite{Bakshiintermediate} for details)  for inputs coming from $P^\prime$ where $T = T_{k_1,\cdots,k_b}^{k_0}$ and $\tilde{T} = \tilde{T}_{\tilde{k}_1,\cdots,\tilde{k}_{\tilde{b}}}^{\tilde{k}_0}$ be tangles with $k_i = \tilde{k}_0$:
	\begin{equation}\label{mainequation}
	Z_{F(T) \circ_i F(\tilde{T})} = \tau(q)^{\frac{1}{2}({k}_i + l(T \circ_i \tilde{T}) - l(T) - l(\tilde{T}))} Z_{F(T \circ_i \tilde{T})}.
	\end{equation}
	Observe that the Equation \ref{mainequation} is obtained by putting $f=1$ in the Equation \ref{cocycle}. Following \cite{Bakshiintermediate}(see Theorem 3.4), the proof of the verification of the Equation \ref{mainequation} with inputs from $P^\prime$ is divided into various steps and subcases until the proof is obvious.
	
	\noindent
	{Step 1:} We reduce to the case that $T$ is a $0_+$ tangle. So the equation that must be seen to hold on $P^\prime $ is
	\begin{equation*}
	Z_{T \circ_i F(\tilde{T})} = \tau(q)^{\frac{1}{2}({k}_i + l(T \circ_i \tilde{T}) - l(T) - l(\tilde{T}))} Z_{T \circ_i \tilde{T}}.
	\end{equation*}
	(since $F(T) =T$ for a $0_+$-tangle $T$). See \cite{Bakshiintermediate} for details.
	
	\noindent
	{Step 2:} As in \cite{Bakshiintermediate}, appealing to sphericality we may assume that the tangle $T$ has the form given in Figure \ref{fig:tangleT} where $\hat{T}$ is some tangle of colour $k_1$ and $i=1$.
	
	\noindent
	{Step 3:} It is sufficient to prove the Equation \ref{mainequation} holds for $\tilde{T}$ being a  Temperley-Lieb tangle as explained in \cite{Bakshiintermediate}.
	
	\noindent 	
	{Step 4:} We settle the Temperley-Lieb case by induction on $k_1$ by dividing into two different subcases. In each of the subcases, we will show that the statement for a suitably chosen $S$ and $\tilde{S}$ with $k_0(\tilde{S})<k_0(\tilde{T})$ implies it for $T$ and $\tilde{T}$. The only case where we use the irreducibility of $N\subset M$ is the subcase 4.2(b), see \cite{Bakshiintermediate} for details. We will modify the proof of this subcase accordingly and all other cases remain the same.

	So without loss of generality, we can assume that $\tilde{T}$ is a Temperley-Lieb tangle such that some $2i$ and $2i+1$ are joined and $T$ is a $0_+$ tangle having the form as in Figure \ref{fig:tangleT} with the black intervals $[2i-1,2i]$ and $[2i+1,2i+2]$ are part of the same black region in $\hat{T}$. 
	
	\begin{figure}[!h]
		\begin{center}
			\psfrag{k0}{$k_1$}
			\psfrag{hatt}{\huge $\hat{T}$}
			\psfrag{1}{\huge $1$}
			\resizebox{3.3cm}{!}{\includegraphics{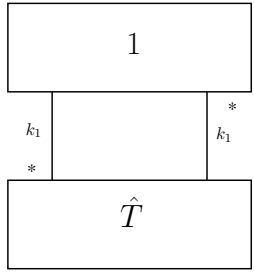}}
		\end{center}
		\caption{The tangle $T$.}
		\label{fig:tangleT}
	\end{figure}
	Now $\tilde{T}$ has the form in Figure \ref{fig:tildet} for some Temperley-Lieb tangle $\tilde{S}$ of colour $k_1-1$. 
	\begin{figure}[!h]
		\begin{center}
			\psfrag{dots}{\huge ${\cdots}$}
			\psfrag{tildeS}{\huge $\tilde{S}$}
			\psfrag{q}{\Huge $q$}
			\psfrag{1}{\large $1$}
			\psfrag{2i-1}{\large $2i-1$}
			\psfrag{2i}{\large $2i$}
			\psfrag{2i+1}{\large $2i+1$}
			\psfrag{2i+2}{\large $2i+2$}
			\psfrag{2k1}{\large $2k_1$}
			\psfrag{=}{\Huge $= \delta \tau(q)$}
			\resizebox{7.5cm}{!}{\includegraphics{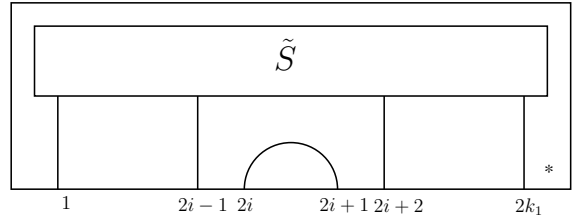}}
		\end{center}
		\caption{Tangle $\tilde{T}$.}
		\label{fig:tildet}
	\end{figure}
	Draw a dotted line from the midpoint of the interval $[2i-1,2i]$ to the midpoint of the interval 
	$[2i+1,2i+2]$ in $\hat{T}$ that lies entirely in the black region that these are both part of. This
	line does not intersect any string of $\hat{T}$ (by the definition of a region) and so the part of $\hat{T}$ that lies inside this dotted line is a 1-box that joins the points $2i$ and $2i+1$. Hence $\hat{T}$ has the form given as in Figure \ref{fig:subcase4.2b} where $W$ is some tangle of colour $k_1-1$.
	\begin{figure}[!h]
		\begin{center}
			\psfrag{dots}{\huge ${\cdots}$}
			\psfrag{1}{\large $1$}
			\psfrag{w}{\huge $W$}
			\psfrag{2i-1}{\large $2i-1$}
			\psfrag{2i}{\large $2i$}
			\psfrag{2i+1}{\large $2i+1$}
			\psfrag{2i+2}{\large $2i+2$}
			\psfrag{2k1}{\large $2k_1$}
			\psfrag{*}{\huge $*$}
			\psfrag{=}{\Huge $= \delta \tau(q)$}
			\resizebox{6.5cm}{!}{\includegraphics{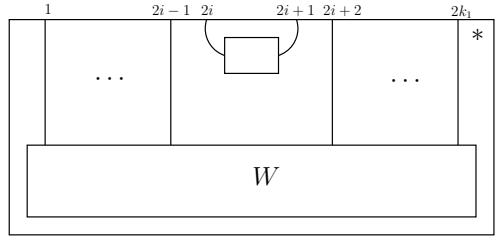}}
		\end{center}
		\caption{Tangle $\hat{T}$.}
		\label{fig:subcase4.2b}
	\end{figure}
	
	Now by assumption we have $Q^\prime \cap M=\mathbb{C}$. This is equivalent to the pictorial relation as in Figure \ref{fig:relation1}. Indeed it follows from Corollary 3.3 and Lemma 4.2 of \cite{BhaLa}.
	\begin{figure}[!h]
		\begin{center}
			\psfrag{q}{\huge $q$}
			\psfrag{x}{\huge $x$}
			\psfrag{tr}{\huge $\tau(q)$}
			\psfrag{*}{\huge$*$}
			\psfrag{dots}{$\cdots$}
			\resizebox{7.8cm}{!}{\includegraphics{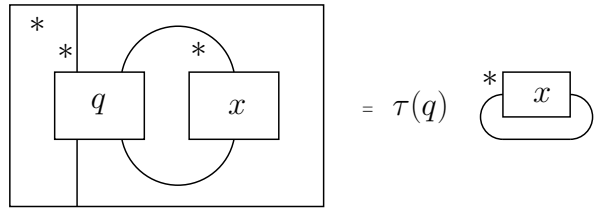}}
		\end{center}
		\caption{Pictorial relation for $Q^\prime \cap M= \mathbb{C}$.}
		\label{fig:relation1}
	\end{figure}
	Furthermore, Figure \ref{fig:relation1} together with the exchange relation in the Figure \ref{fig:biprojection}(d) implies the pictorial relation in Figure \ref{fig:relation2} for $x \in P_1$.
	
	\begin{figure}[!h]
		\begin{center}
			\psfrag{q}{\huge $q$}
			\psfrag{x}{\huge $x$}
			\psfrag{tr}{\huge $\tau(q)$}
			\psfrag{*}{\huge$*$}
			\psfrag{=}{\huge $=$}
			\psfrag{dots}{$\cdots$}
			\resizebox{9.5cm}{!}{\includegraphics{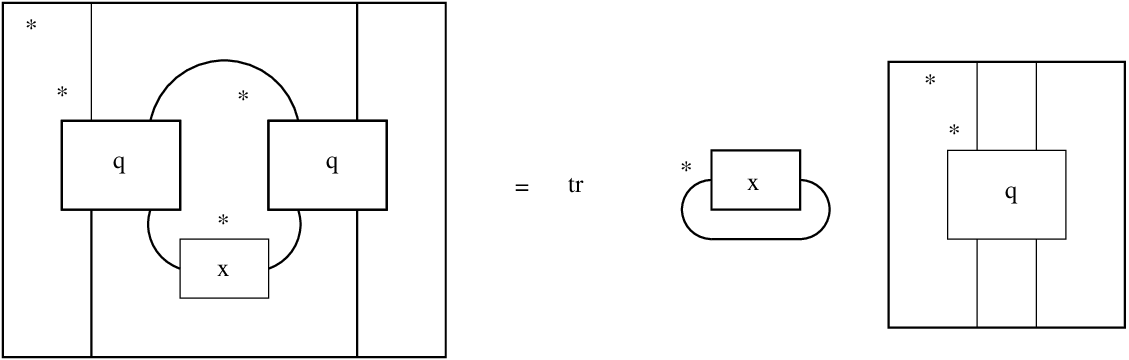}}
		\end{center}
		\caption{Modified relation in the non irreducible case.}
		\label{fig:relation2}
	\end{figure}

	\underline{Claim:} Let $S$ be the tangle given in Figure \ref{fig:stangle}. We claim that the validity of the Equation \ref{mainequation} for the pair $(S,\tilde{S})$ implies the validity for the pair $(T,\tilde{T})$.
	\begin{figure}[!h]
		\begin{center}
			\psfrag{w}{\huge $W$}
			\psfrag{*}{\huge $*$}
			\psfrag{k}{\huge $k_i-1$}
			\psfrag{1}{\huge $1$}
			\resizebox{3.6cm}{!}{\includegraphics{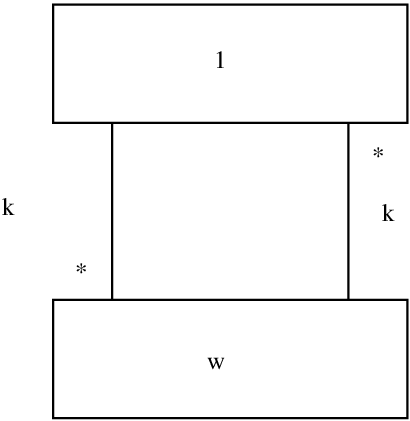}}
		\end{center}
		\caption{Tangle $S$.}\label{fig:stangle}
	\end{figure}

	So suppose that Equation \ref{mainequation} holds for the pair $(S,\tilde{S})$. Then we have
	\begin{equation}\label{main2}
	Z_{S \circ_1 F(\tilde{S})} = \tau(q)^{\frac{1}{2}({k}_1-1 + l(S \circ_1 \tilde{S}) - l(S) - l(\tilde{S}))} Z_{S \circ_1 \tilde{S}}. 
	\end{equation}
The tangles $T\circ_1\tilde{T}$ and $S\circ_1\tilde{S}$ are shown in Figure \ref{fig:tcomposettilde}.
\begin{figure}[!h]
	\begin{center}
		\psfrag{w}{\huge $W$}
		\psfrag{*}{\huge $*$}
		\psfrag{dots}{\huge $\cdots$}
		\psfrag{k-1}{\huge $k_1-1$}
		\psfrag{tildes}{\huge $\tilde{S}$}
		\psfrag{T}{\huge $T\circ_1\tilde{T}$}
		\psfrag{S}{\huge $S\circ_1\tilde{S}$}
				\resizebox{11cm}{!}{\includegraphics{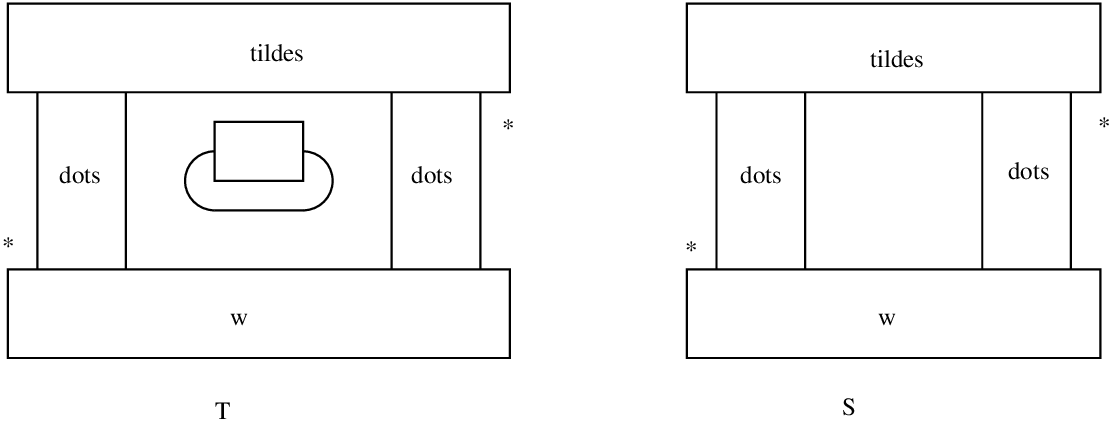}}
				\caption{The tangles $T\circ_1\tilde{T}$ and $S\circ_1\tilde{S}$. }
				\label{fig:tcomposettilde}
	\end{center}
\end{figure}
	It can be easily seen that $T\circ_1 \tilde{T}$ has an extra floating one box than $S \circ_1 \tilde{S}$. Hence $Z_{T\circ_1 \tilde{T}}=$ $\begin{minipage}{.05\textwidth}
	\centering
	\psfrag{x}{$x$}
	\includegraphics[scale=.45]{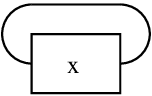}
	\end{minipage}
	\qquad Z_{S \circ_1 \tilde{S}}$ for $x \in P_1$ and $l(T\circ_1 \tilde{T})=l(S \circ_1 \tilde{S})+1$.
	Also we have,
	$l(T)=l(\hat{T})=l(W)=l(S)$ and $l(\tilde{T})=l(\tilde{S})$.
	Now it remains to compare $T\circ_1F(\tilde{T})$ and $S\circ_1F(\tilde{S})$ given in Figure \ref{fig:tftilde}.
	\begin{figure}[!h]
		\begin{center}
			\psfrag{w}{\huge $W$}
			\psfrag{dots}{\huge $\cdots$}
			\psfrag{T}{\huge $T\circ_1F(\tilde{T})$}
		\psfrag{S}{\huge $S\circ_1F(\tilde{S})$}
			\psfrag{q}{\huge $q$}
			\psfrag{tildeS}{\huge $\tilde{S}$}
			\resizebox{13cm}{!}{\includegraphics{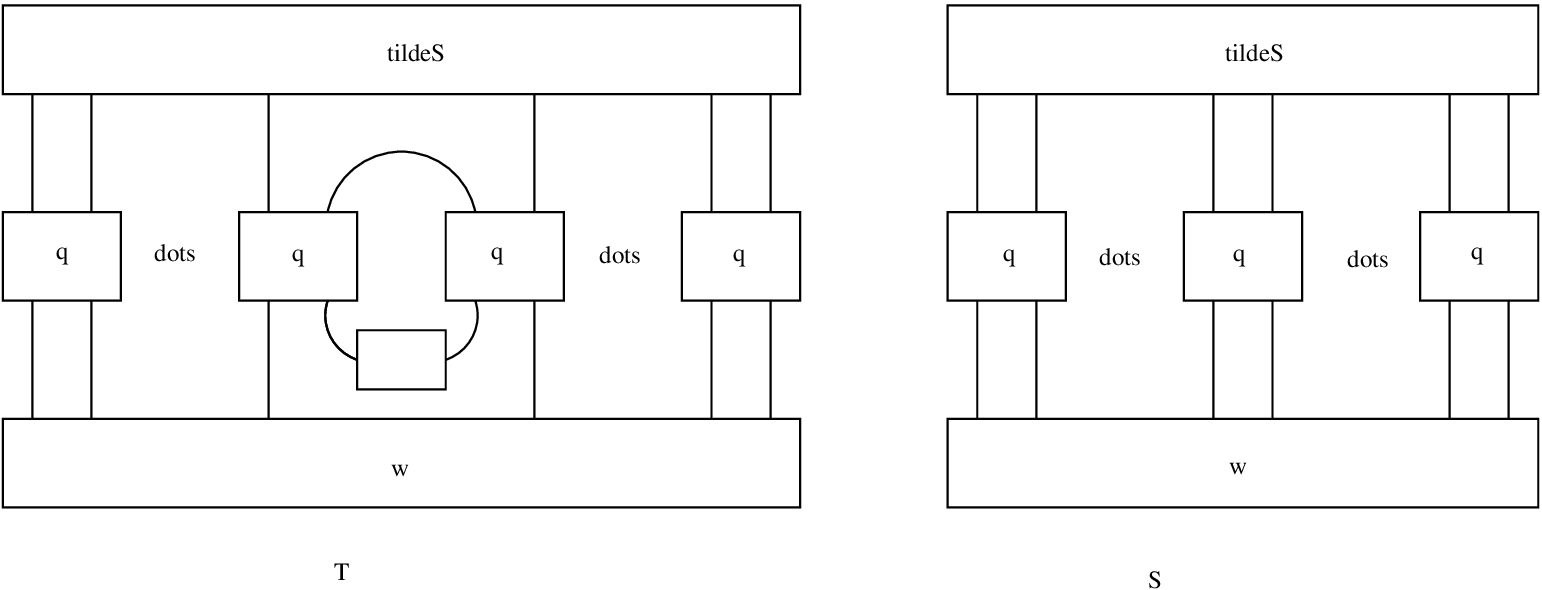}}
		\end{center}
		\caption{The tangles $T\circ_1F(\tilde{T})$ and $S\circ_1F(\tilde{S})$.}\label{fig:tftilde}
	\end{figure}
	Applying the relation in Figure \ref{fig:relation2} we have
	$
	Z_{T\circ_1F(\tilde{T})}=\tau(q)~\begin{minipage}{.05\textwidth}
	\centering
	\psfrag{x}{$x$}
	\includegraphics[scale=.45]{tracex.eps}
	\end{minipage}
	\qquad Z_{S\circ_1F(\tilde{S})}$
	for $x \in P_1$ and this together with the Equation \ref{main2} justifies the claim.

Combining all the steps we may conclude that $(P^{\prime}, Z_{T^{\prime}})$ is a planar algebra. To finish the proof of Case (I) we observe that the planar algebra $(P^{\prime}, Z_{T^{\prime}})$ is isomorphic to the planar algebra $P^{N\subset Q}$ by Theorem 3.7 of \cite{Bakshiintermediate}.

	Case (II): $Q^\prime \cap M\neq \mathbb{C}$. \color{black}Let $f\in Q^\prime \cap M$ be a minimal projection. Consider the intermediate $Nf\subset Qf\subset fMf$ with corresponding biprojection $\frac{1}{\text{tr}f}f_2qf_2$, see Lemma \ref{biprojection}. Observe that 
	\begin{eqnarray*}
		\alpha_{Nf\subset Qf\subset fMf}(T)=[fMf:Qf]^{\frac{1}{2}c(T)}&=&
		(\text{tr}f)^{c(T)}[M:Q]^{\frac{1}{2}c(T)}\\&=&(\text{tr}f)^{c(T)}\alpha_{N\subset Q\subset M}(T).
	\end{eqnarray*}
	Now since $f$ is minimal, by Lemma \ref{minimalprojection}, $(Qf)^\prime \cap fMf=f(Q^\prime \cap M)f=\mathbb{C}$. So this reduces to Case (I). By Lemma \ref{subfactoriso}, we have the subfactor $Nf\subset Qf$ isomorphic to $N\subset Q$. 
	Hence we have,
	\begin{eqnarray*}  
		P^{N\subset Q}_n\cong P^{Nf\subset Qf}_n&=&F_n(P^{Nf\subset fMf}_n)\qquad (\text{by Case (I)})\\
		&=& F_n((fPf)_n)\qquad \quad(\text{reduced planar algebra, $\S 2$})\\
		&=& F_n(f_nP_nf_n)\\
		&=& P_n^\prime
	\end{eqnarray*}
	and
	\begin{eqnarray*} 
		Z_T^P\cong Z_T^{P^{Nf\subset Qf}}&=&\alpha_{\scaleto{Nf\subset Qf\subset fMf}{5pt}}(T) Z^{P^{Nf\subset fMf}}_{F(T)}|_{P_n^{Nf\subset Qf}}\quad (\text{by Case (I)})\\
		&=& (\text{tr}f)^{c(T)}\alpha_{N\subset Q\subset M}(T)Z^P_{fF(T)f|{P^\prime}}\quad(\text{reduced planar algebra, $\S 2$})\\
		&=&(\text{tr}f)^{c(T)}\alpha_{N\subset Q\subset M}(T)Z^P_{F(fTf)|{P^\prime}}\quad (\text{follows from Remark \ref{example}\color{black}})\\
		&=& Z_T^{P^\prime}.	
	\end{eqnarray*}
This completes the proof of the theorem.
\end{proof}
\begin{remark}
	The planar algebra of $Q\subset M$ may also be described in terms of planar algebra of $N\subset M$ by considering the `dual planar algebra' $P^{M\subset M_1}$ (see \cite{Joplanar1}) since the type $II_1$ factor $Q_1$ is an intermediate subfactor of $M\subset M_1$. The details are easy and omitted. \end{remark}
\section{acknowledgement}
The authors would like to thank Prof. Vijay Kodiyalam for various useful discussions. The first named author was supported through DST INSPIRE Faculty grant (reference no. DST/INSPIRE/04/2019/002754). The authors would like to thank the referee for pointing out the related references \cite{hartglass17} and \cite{Bi94}.
\color{black}
\nocite{*}
\bibliographystyle{plain}
\bibliography{intermediate}
\end{document}